\documentclass[default]{sn-jnl}

\usepackage{comment}
\usepackage[normalem]{ulem}
\usepackage{amsmath}
\usepackage{bussproofs}
\usepackage{multicol}
\usepackage{multirow}

\allowdisplaybreaks[1]

\newtheoremstyle{mystyle1}
  {}
  {}
  {\normalfont}
  {}
  {\bfseries}
  {}
  {10pt}
  {}
\theoremstyle{mystyle1}
\newtheorem{definition}{Definition} 

\newtheoremstyle{mystyle2}
  {}
  {}
  {\itshape}
  {}
  {\bfseries}
  {}
  {10pt}
  {}
\theoremstyle{mystyle2}
\newtheorem{proposition}[definition]{Proposition} 
\newtheorem{corollary}[definition]{Corollary} 
\newtheorem{lemma}[definition]{Lemma} 

\newcommand{\sem}[1]{[ \! [ #1 ] \! ]} 
\DeclareMathOperator{\SF}{\normalfont{\textbf{SF}}}
\DeclareMathOperator{\LK}{\normalfont{\textbf{LK}}}
\DeclareMathOperator{\CPC}{\normalfont{\textbf{CPC}}}
\DeclareMathOperator{\IPC}{\normalfont{\textbf{IPC}}}

\DeclareMathOperator{\HT}{\normalfont{\textbf{HT}}}
\DeclareMathOperator{\fin}{fin}
\DeclareMathOperator{\Fm}{Form}
\DeclareMathOperator{\ClForm}{ClForm}

\DeclareMathOperator{\NDE}{NDE}

\DeclareMathOperator{\Eq}{Eq}

\DeclareMathOperator{\ST}{ST}
\DeclareMathOperator{\ClST}{ClST}
\DeclareMathOperator{\Var}{Var}
\DeclareMathOperator{\cl}{cl}

\begin{document}

\setcounter{tocdepth}{7}
\setcounter{secnumdepth}{5} 

\title[Wright's Strict Finitistic Propositional Logic]{Wright's Strict Finitistic Logic in the Classical Metatheory: The Propositional Case}

\author*{\fnm{Takahiro} \sur{Yamada}}\email{g.yamadatakahiro@gmail.com}

\affil{\orgdiv{The Department of Philosophy and Religious Studies}, \orgname{Utrecht University}, \orgaddress{\street{Janskerkhof 13}, \city{Utrecht}, \postcode{3512 BL}, \state{Utrecht}, \country{the Netherlands}}}

\abstract{Crispin Wright in his 1982 paper argues for strict finitism, a constructive standpoint that is more restrictive than intuitionism. In its appendix, he proposes models of strict finitistic arithmetic. They are tree-like structures, formed in his strict finitistic metatheory, of equations between numerals on which concrete arithmetical sentences are evaluated. As a first step towards classical formalisation of strict finitism, we propose their counterparts in the classical metatheory with one additional assumption, and then extract the propositional part of `strict finitistic logic' from it and investigate. We will provide a sound and complete pair of a Kripke-style semantics and a sequent calculus, and compare with other logics. The logic lacks the law of excluded middle and Modus Ponens and is weaker than classical logic, but stronger than any proper intermediate logics in terms of theoremhood. In fact, all the other well-known classical theorems are found to be theorems. Finally, we will make an observation that models of this semantics can be seen as nodes of an intuitionistic model.}

\keywords{strict finitism, Crispin Wright, constructivism, finitism, classical reconstruction}

\def\fCenter{\ \vdash\ } 

\maketitle

\section{Introduction}

\subsection{Aims}\label{section: Aims}
The present paper provides and explores the propositional part of `strict finitistic logic' according to Wright, as obtained via classical counterparts of his strict finitistic models of arithmetic, under an additional assumption we call the `atomic prevalence condition'. `Strict finitism' is the view of mathematics according to which an object, or a number in particular, is admitted iff it is constructible in practice, and a statement holds iff it is verifiable in practice. It is constructivist, in that it accepts a number and a statement on grounds of our cognitive capabilities; and more restrictive than intuitionism, since it uses the notion of `in practice' in place of intuitionism's `in principle'. Strict finitism is finitistic, because as a consequence, it rejects the idea that there are infinitely many natural numbers. Strict finitistic logic is meant to be the abstract system of logical reasoning concerning actually constructible objects, based on actual verifiability.

Among the literary sources, Wright's in 1982 \cite{Wright1982}, to us, is the most philosophically inspirational and appears to be in possession of most formal-logical contents. While \cite{Wright1982} is a philosophical rebuttal to Dummett's criticism \citep{Dummett1975} against strict finitism, its appendix contains a proposal of the models of strict finitistic arithmetic\footnote{
We must note \cite{Yessenin-Volpin1970} as an important precursor to Dummett's and Wright's strict finitism.
}.
Wright sketches an `Outline of a strict finitist semantics for first-order arithmetic' \cite[p.167]{Wright1982}, and describes notions not unlike the models and the forcing conditions of the intuitionistic Kripke semantics. A `strict finitistic tree' is a model of strict finitistic arithmetic which represents practically possible histories of an ordinary subject's arithmetical development. Each node stands for a stage of construction, and is assigned the numbers (numerals e.g. $\overline{0} + \overline{0}$) and the arithmetical truths (equations e.g. $\overline{0} + \overline{0} = \overline{0}$) the agent has actually constructed and learnt. The `verification-conditions' \cite[p.169]{Wright1982} determine what formulas in the language of first-order arithmetic are forced at a node, based on the assigned truths as atoms.

Our interest lies in what kind of sentence is valid according to Wright's semantics, i.e. forced at every node in every strict finitistic tree. While there has been a certain amount of research that involves elucidation of the structure of the strict finitistic numbers\footnote{
E.g., see \cite[pp.474-5]{Magidor2012} and \cite[pp.314-5]{Dean2018}.
},
we hope, by our research, we could set foot in a path to understanding strict finitistic reasoning.

\subsection{Methods}\label{section: Methods}
Wright introduced the trees and the conditions in his strict finitistic metatheory. This causes some problems. Firstly, further mathematical investigations are hindered. Strict finitism rejects mathematical induction and the law of excluded middle, the metalevel principles allowed in the classical metatheory. Secondly, it is not clear how to interpret the strict finitistic trees. They are structures `small enough to practically construct'. But this notion can be classically inconsistent. At least the following are supposed to hold\footnote{
We will touch upon a fourth condition and see how it is realised in our system, after proposition \ref{proposition: Assertibility and validity calculations}.
}.
\begin{enumerate}
  \item $\overline{0}$ is a practically constructible numeral.
  \item If a numeral is practically constructible, then so are all of its direct successors.
  \item There is a numeral that is not practically constructible.
\end{enumerate}
So some path of a strict finitistic tree cannot be regarded as isomorphic to an initial segment of $\mathbb{N}$, since these conditions would entail a contradiction. Indeed, these are common assumptions of strict finitism and the main reason why formalising it has been a challenge.

With these considerations, we will adopt the classical metatheory to investigate in, and consider classical structures with some constraints as our counterparts of the strict finitistic trees, rather than try to render them. Our `models of arithmetic' will be intuitionistic Kripke frames with an assignment of arithmetical sentences, which do not involve variables. After investigating them, we define and explore an abstract, Kripke-style `strict finitistic semantics' to capture the validity in them schematically. The models in this system, the `strict finitistic models', deal with formulas built from variables which can stand for concrete sentences via substitution. The same kinds of constraint will be applied, and all semantic properties are then carried over. We will then define a sequent calculus, and prove it to be sound and complete with regard to the strict finitistic semantics.

We will maintain Wright's forcing conditions. His condition of negation (and that of implication, section \ref{subsection: Characteristics (i): Strict finitistic implication}) will then affect how we interpret the forcing relation. Let us write $k \models A$ for a sentence or formula $A$ holding at a node $k$ of a tree-like structure we consider. The condition is $k \models \neg A$ iff $l \not \models A$ for any node $l$. Thus strict finitistic negation stands for practical unverifiability. Also, it is `global' in that if $\neg A$ holds somewhere in the structure, it holds everywhere. Therefore $A$ of $k \models \neg A$ may involve an expression the agent has not constructed, and as a result $k \models \neg A$ may not represent the agent's judgement at step $k$. Rather a safer interpretation would be that $k \models A$ represents the agent's knowledge at $k$ if $A$ is atomic; otherwise (or in all cases), it only displays our judgement as the classically-thinking model-maker.\footnote{
We thank an anonymous referee for suggesting that therefore, some might say that this article's logic is the logic of our judgement about practical constructibility, not the logic of strict finitistic reasoning itself. It would indeed be a sensible reaction, and we agree that it might be the case. The genuine `logic of strict finitistic reasoning' may be the one that captures an ordinary subject's judgement about their own practical constructibility. It may be Wright's original semantics in the strict finitistic metatheory. It is a model-theoretic approach to understanding an ordinary subject's assertion, e.g., that a statement is unverifiable to themselves. The models there are strict finitistic themselves, since they need be e.g. `small enough'. We hope that this article contributes to making the semantics understandable to us, classical thinkers, by providing its classical counterpart.
}

We assume the following `finitistic' constraints on our models of arithmetic.
\begin{enumerate}
  \item The frame is finite.
  \item Each node forces at most finitely many atomic sentences (the finite verification condition).
\end{enumerate}
Each model stands for one collection of the entire possibilities of an agent's actual construction and verification. Our models will be finite by (i), but indeed include ones that are not `small enough'. We accept them as part of our classical idealisation, in hope that the totality of the valid sentences is not affected when seen schematically.

We will assume an auxiliary condition. We say a sentence $A$ is `assertible' in a model of arithmetic if $k \models A$ for some $k$. A sentence $B$ is `prevalent' if
\begin{itemize}
  \item for any node $k$, there is a $k' \geq k$ such that $k' \models B$, where $\leq$ is the partial order on the structure\footnote{
`Assertibility' is Wright's terminology \cite[p.170]{Wright1982}, while `prevalence' is ours.
}.
\end{itemize}
Assertibility is practical verifiability in a weaker sense, while prevalence is stronger, since $A$ is verified in a case and $B$ is eventually verified in any case. A model of arithmetic has the `atomic prevalence property' if all assertible, atomic sentences are prevalent. Throughout this article, we will only consider ones with this property. This is a strong assumption that collapses the two notions, and indeed, Wright noted that we cannot generally assume it \cite[pp.168-9]{Wright1982}. There is no guarantee that a truth found in one history is found in every other history, because an agent's construction power would reduce as the construction continues. A truth may require so tremendous an amount of power, that they cannot discover it after obtaining some others. However, to make our study feasible, we will restrict the scope of discussion. We will touch upon this issue in the ending remarks. 

We will also use, only as a tool, a structure $S_\infty$ with $\omega$-sequences as its branches. Our models of arithmetic will be defined to be $S_\infty$'s finite parts. Certainly, it does not satisfy the `finitistic' constraints. But it would appear to be a natural extension of the models, has the atomic prevalence property, and does not affect the totality of the valid sentences (proposition \ref{proposition: Model contraction in S}).

\subsection{Characteristics (i): Strict finitistic implication}\label{subsection: Characteristics (i): Strict finitistic implication}
Strict finitistic implication $A \to B$ means that $B$ holds after $A$ sooner or later. We translate Wright's condition as that $k \models A \to B$ iff for any node $k' \geq k$, if $k' \models A$, then $k'' \models B$ for some node $k'' \geq k'$. This is a version of intuitionistic implication that allows for a time-gap between $A$ and $B$. As we see it, Wright did not restrict the gap, since every strict finitistic tree is supposed to be small enough. As part of our classical idealisation, we accept any finite gap (cf. the ending remarks).\footnote{
In the style of the BHK interpretation, $A \to B$ might mean the existence of a method of transforming any verification of $A$ into one for $B$ within a `small enough' number of steps. This is a version of intuitionistic implication that respects the time scale of transformation.
}

Our implication generalises the atomic prevalence property to all complex sentences (proposition \ref{proposition: Prevalence property}). Under prevalence, an assertible, atomic sentence describes a fact in the realm of the foreseeable future, which figuratively is a `candy bin' of a jumble of facts, regardless of the order of their verification. Implication's condition matches this idea: $A \to B$ means that if $A$ holds in the future, so does $B$, regardless of which comes first. As a result, assertibility and prevalence are equivalent in general, and stand for practical verifiability equally.

This idea of `order-less truths' may suggest classical logic. In fact, assertibility can be calculated just as the classical truth-values (proposition \ref{proposition: Assertibility and validity calculations}), and assertibility in all strict finitistic models coincides with provability in classical propositional logic \textbf{CPC} (corollary \ref{corollary: Assertibility is CPC})\footnote{
This conforms to the idea the strict finitist would maintain that all and only the practically verifiable statements are properly subject to classical logic. Cf. e.g., \cite[p.115]{Wright1982}.
}.
Accordingly, most of the classical tautologies (and rules) are found valid (propositions \ref{proposition: Assertibility and validity calculations}, \ref{proposition: HT subseteq SF}). Exceptions are the law of the excluded middle and Modus Ponens (proposition \ref{proposition: Failure of LEM and MP}).

With generalised prevalence, we can aptly treat $\neg A$ as an abbreviation of ${A \to \bot}$, where $\bot$ stands for a statement practically unverifiable. Formally, the prevalence property implies that $k \models A \to \bot$ iff $l \not \models A$ for all $l$ (cf. proposition \ref{proposition: Prevalence property}). Conceptually, too, $A \to \bot$ and $\neg A$ coincide. $A \to \bot$ means that if $A$ eventually holds, then so does $\bot$. Since the consequent is impossible, it means that $A$ does not eventually hold, i.e., is practically unverifiable.

\subsection{Characteristics (ii): Relation with intuitionism}\label{subsection: Characteristics (ii): Relation with intuitionism}
Some take the following as the conceptual relation between strict finitism and intuitionism:
\begin{itemize}
  \item a statement is verifiable in principle iff it is verifiable in practice with some finite extension of practical resources\footnote{
Wright brings up an identification scheme of `decidability in principle' along this line \cite[p.113]{Wright1982}. Also \cite[p.147]{Tennant1997} and \cite[p.278]{Murzi2010} suppose a relation of this kind.
}.
\end{itemize}
We will prove results that seem to formalise this identification scheme. Strict finitistic models with finite frames, arranged in ascending order of practical verification power, can be viewed as the nodes of an intuitionistic model, and vice versa. We will formalise this order using the notion of `generations'. Let $\mathcal{U}$ be a set of strict finitistic finite models. Then the set $A_W$ of the variables assertible in $W \in \mathcal{U}$ determines the assertibility, and hence the practical verifiability of every formula in $W$. Let us write $W \preceq W'$ if $W$ is an initial part of another model $W'$ and $A_W \subseteq A_{W'}$. Then $W \preceq W'$ means that $W'$ represents the same agent as $W$, but with more verificatory resources, typically from a later generation.

We will call $G := \langle \mathcal{U}, \preceq \rangle$ a `generation structure' (`g-structure'). This is an intuitionistic frame with strict finitistic finite models as its nodes, and thereby serves as a framework to express practical verifiability with increasing power. Here two evaluations on $G$ are naturally induced, one intuitionistic, and the other in the strict finitistic style. Since $A_W$ summarises practical verifiability in $W$ as a node of $G$, we define a valuation $v$ by $W \in v(p)$ iff $p \in A_W$. Then $v$ defines an intuitionistic forcing relation $\Vdash$. Also a new evaluation, the `generation forcing relation' $\Vvdash$, is defined between the pairs of a $W \in \mathcal{U}$ and a node of $W$ and the formulas, using the valuation in each $W \in \mathcal{U}$ for the base case. Implication will be defined so that each $W$ is closed under it:
\begin{itemize}
  \item $W, k \Vvdash A \to B$ iff for any $W' \succeq W$ and node $k'$ of $W'$, if $k' \geq k$ and $W', k' \Vvdash A$, then for some $k''$ of $W'$, $k'' \geq k'$ and $W', k'' \Vvdash B$.
\end{itemize}
$W, k \Vvdash A \to B$ means no matter how power will increase, $B$ comes after $A$ `in a short while', within the same generation; and hence $\Vvdash$ thus defined is a strict finitistic forcing relation that takes into account the increasement in power. Validity according to $\Vvdash$ is being forced at all pairs. We will show that an intuitionistic model with the finite verification condition induces a g-structure; and the induced evaluations $\Vvdash$ and $\Vdash$ always correspond (propositions \ref{proposition: G to J transformation}, \ref{proposition: J to G transformation}). Further, we will prove that validity in all g-structures coincides with provability in intuitionistic propositional logic \textbf{IPC} (proposition \ref{proposition: Validity in G is IPC}).

These results may formalise the aforementioned conceptual relation. Any statement verifiable in principle is verified in practice within finitely many steps with some extension of power; and the totality of the statements verifiable in principle coincides with that of those in practice under extension of power.

\subsection{The structure of the paper and comparison results}
In section \ref{section: The languages}, we define two languages $\mathcal{L}_{\mathcal{S}}$ and $\mathcal{L}_{\SF}$. In \ref{section: The models of arithmetic}, we define $S_\infty$ with regard to $\mathcal{L}_{\mathcal{S}}$ and our models of arithmetic via $S_\infty$; and look into their semantic properties. The formulas in this context are concrete sentences. To capture validity schematically, we will in section \ref{section: The strict finitistic Kripke semantics} use $\mathcal{L}_{\SF}$, a standard language of propositional logic with variables, to define and study the strict finitistic semantics. In section \ref{section: A complete sequent calculus SF}, we define a sequent calculus $\SF$ and show that it is sound and complete with respect to the strict finitistic semantics. Our completeness proof is in the usual Henkin-style.

In section \ref{section: Relations with intermediate logics} we conduct a comparison with intermediate logics. Beside the relation mentioned right above, we will prove that our logic is weaker than \textbf{CPC}, but stronger than any proper intermediate logic -- in this context, we identify each logic with the set of their theorems, and write \textbf{SF} for strict finitistic logic. An intermediate logic is obtained by adding \textbf{CPC}'s theorems to \textbf{IPC} and taking the closure under Modus Ponens and substitution of formulas. A resulting logic is `proper' if it is not \textbf{CPC}. It is known that the intermediate logics form a lattice under the set-theoretic inclusion, and it is bounded by $\CPC$ and $\IPC$ as its maximum and minimum, respectively. $\CPC$ has a unique direct predecessor called the `logic of here-and-there' \textbf{HT}, characterised by the class of the intuitionistic models with only $2$ nodes (`here' and `there')\footnote{
It is axiomatised by $A \lor (A \to B) \lor \neg B$; and also called G\"{o}del 3-valued logic and Smetanich's logic. For a modern treatment and application of \textbf{HT}, see e.g. \cite{Lifschitz2001}.
}.
We will show $\HT \subsetneq \textbf{SF} \subsetneq \textbf{CPC}$ (corollary \ref{corollary: Assertibility is CPC}, proposition \ref{proposition: HT subseteq SF}) \footnote{
\textbf{SF} is not itself an intermediate logic, since it lacks Modus Ponens.
}.
Then we move on to how strict finitistic models can be seen as nodes of an intuitionistic model; and end this article with some remarks in section \ref{section: Ending remarks: Further topic}.

\section{Strict finitistic propositional logic}\label{section: Strict finitistic propositional logic}

\subsection{The languages}\label{section: The languages}
We use two languages, $\mathcal{L}_{\SF}$ and $\mathcal{L}_{\mathcal{S}}$. $\mathcal{L}_{\SF}$ is the language of strict finitistic logic. It is a standard language of propositional logic: $\Var$ is the countable set of variables, and we define the set $\Fm$ of all formulas in $\mathcal{L}_{\SF}$ by
\begin{itemize}
  \item $\Fm ::= \bot \mid \Var \mid \Fm \land \Fm \mid \Fm \lor \Fm \mid \Fm \to \Fm$.
\end{itemize}
We use $\mathcal{L}_{\mathcal{S}}$ to describe our models of strict finitistic arithmetic. It is a language of quantifier-free arithmetic with no variables. We deal with numbers through their expressions, and they are the terms in $\mathcal{L}_{\mathcal{S}}$. We call them after Wright \cite[p.167]{Wright1982} `natural-number denoting expressions' (or `NDEs')\footnote{
Wright \cite[p.167]{Wright1982} writes `nde' for the singular and `nde's' for the plural form.
}.
\begin{itemize}
  \item $\NDE ::= \overline{0} \mid S(\NDE) \mid \NDE + \NDE \mid \NDE \cdot \NDE$.
\end{itemize}
Arithmetical propositions are expressed by the formulas in $\mathcal{L}_{\mathcal{S}}$. The atoms are $\bot$ and the `equations' between NDEs. We let $=$ be the only predicate symbol, and $\Eq = \{ x = y \mid x, y \in \NDE \}$. All formulas in $\mathcal{L}_{\mathcal{S}}$ are closed, and we define their set $\ClForm$ by
\begin{itemize}
  \item $\ClForm \! ::= \! \bot \mid \Eq \mid \ClForm \land \ClForm \mid \ClForm \lor \ClForm \! \mid \! \ClForm \! \to \! \ClForm$.
\end{itemize}
In both, we let $\neg A$ be an abbreviation of $A \to \bot$. We note that neither has $\top$.

\subsection{The models of arithmetic}\label{section: The models of arithmetic}
In this section we provide our models of strict finitistic arithmetic, which are counterparts of Wright's strict finitistic trees. Our main focus is on what kind of sentence is valid in them. To see it, we will define a structure we call $S_\infty$. Having it among the models does not change what is valid in all models (proposition \ref{proposition: Model contraction in S}), and the models' properties will be plain from seeing it.

$S_\infty$ is one fixed rooted tree-like structure, where NDEs and equations are assigned to the nodes. $S_\infty$ is a tuple $\langle T, \leq_T, M, E  \rangle$, where $\langle T, \leq_T \rangle$ is a finitely-branching tree each of whose branches is an $\omega$-sequence; $M$ is a mapping from $T$ to $\mathcal{P}(\NDE)$; $E$ is one from $T$ to $\mathcal{P}(\Eq)$. Conceptually, $M(t)$ with $t \in T$ (ideally) stands for the set of the NDEs the agent has actually constructed by stage $t$; $E(t)$ the set of the equations, the arithmetical facts, they have actually verified and learnt by $t$. Letting $t_0$ be the root, we set $M(t_0) = \{ \overline{0} \}$ and $E(t_0) = \emptyset$; and for each $t \in T$ and $x, y \in M(t)$,
\begin{enumerate}
  \item if $S(x) \notin M(t)$, then we let there be a direct successor $t'$ such that $M(t') = M(t) \cup \{ S(x) \}$ and $E(t') = E(t)$;
  \item similar for $x+y$ and $x \cdot y$; and
  \item if $\sem{x} = \sem{y}$ and $x = y \notin E(t)$, where $\sem{z}$ is the natural number that $z \in \NDE$ denotes, then we let there be a direct successor $t'$ such that $M(t') = M(t)$ and $E(t') = E(t) \cup \{ x = y \}$.
\end{enumerate}
\noindent (We assume $\sem{ \, }$ in the background in the natural way.) So the agent proceeds one by one. $S_\infty$ is not officially a model of arithmetic, but it serves as a useful tool for our investigation.

Closed formulas in $\mathcal{L}_{\mathcal{S}}$ are evaluated at each node of $S_\infty$. We interpret Wright's forcing conditions as follows\footnote{
Wright calls them `verification-conditions' in the appendix of \cite{Wright1982}. In section 3 of \cite{Wright1982}, he analyses the notion of actual verification under the name of `verification'.
}.
\begin{definition}[Actual verification conditions]\index{actual verification condition}\label{Actual verification conditions}
  Let $t \in T$. For any $p \in \Eq$, $t \models_{S_\infty} p$ iff $p \in E(t)$. For any $A, B \in \ClForm$,
  \begin{enumerate}
    \item $t \models_{S_\infty} A \land B$ iff $t \models_{S_\infty} A$ and $t \models_{S_\infty} B$;
    \item $t \models_{S_\infty} A \lor B$ iff $t \models_{S_\infty} A$ or $t \models_{S_\infty} B$; and
    \item $t \models_{S_\infty} A \to B$ iff for any $t' \geq_T t$, if $t' \models_{S_\infty} A$, then there is a ${t'' \geq_T t'}$ such that $t'' \models_{S_\infty} B$.
  \end{enumerate}
\end{definition}
\noindent We may omit subscripts in general, if no confusion arises. We note that negation behaves intuitionistically: $t \models_{S_\infty} \neg A$ iff $t' \not \models_{S_\infty} A$ for all $t' \geq_T t$.

We note that this system deals with actual construction of terms as well as verification of sentences, unlike the abstract semantic system defined in section \ref{section: The strict finitistic Kripke semantics}. $t \models_{S_\infty} A \to B$ means that if $A$ is actually verified, then so will $B$. But for $A \to B$ to hold, the agent does not have to know this. If $A \in \Eq$, $t \models_{S_\infty} A$ means that the agent has verified $A$ by, and knows $A$ at step $t$. Otherwise it only reflects our judgement about the agent's verification (cf. section \ref{section: Methods}).

We define validity in $S_\infty$ with two additional notions, assertibility and prevalence. The latter two stand for practical verifiability (cf. section \ref{section: Methods}).
\begin{definition}[Validity, assertibility, prevalence]\label{definition: Validity, assertibility, prevalence}
  An $A \in \ClForm$ is \textit{valid in $S_\infty$} if $t \models A$ for all $t \in T$; and \textit{assertible in $S_\infty$} if $t \models A$ for some $t \in T$. An $x \in \NDE$ is \textit{prevalent in $S_\infty$} if, for any $t \in T$, there is a $t' \geq_T t$ such that $x \in M(t')$. An $A \in \ClForm$ is \textit{prevalent in $S_\infty$} if for any $t \in T$, there is a $t' \geq_T t$ such that $t' \models A$. We write $\models_{S_\infty}^V A$, $\models_{S_\infty}^A A$ and $\models_{S_\infty}^P A$ for that $A$ is valid, assertible and prevalent in $S_\infty$, respectively.
\end{definition}
\noindent Accordingly, $\not \models^V \overline{0} = \overline{0}$, since $t_0 \not \models \overline{0} = \overline{0}$. We are supposing that the ideal agent of $S_\infty$ does not know the equation at the beginning. We will soon introduce models of arithmetic (as elements of $\mathcal{S}_{\fin}$) with various initial knowledge. The notion of semantic consequence is defined via that of validity.
\begin{definition}[Semantic consequence]
  Let $\Gamma, \Delta \subseteq_{\fin} \ClForm$. $\Delta$ is a \textit{semantic consequence of $\Gamma$ in $S_\infty$} if, for any $t \in T$, if $t \models A$ holds for all $A \in \Gamma$, then $t \models B$ for some $B \in \Delta$. We write $\Gamma \models_{S_\infty}^V \Delta$ for this.
\end{definition}
\noindent We do not define this notion for infinite sets, since if $\Gamma$ e.g. contains infinitely many equations, then $\Gamma \models^V \Delta$ would hold vacuously.

Plainly $\models_{S_\infty}$ persists\footnote{
Wright had to posit it and called the `unproved lemma' \cite[p.170, p.172]{Wright1982}.
}.
$S_\infty$ has what we call the \textit{prevalence property}: assertibility implies prevalence.
\begin{proposition}[Prevalence property]\label{proposition: Prevalence property}
  (i) Every NDE is prevalent in $S_\infty$. (ii) $\models_{S_\infty}^A A$ implies $\models_{S_\infty}^P A$.
\end{proposition}
\begin{proof}
  Both by induction. (ii) will use (i).
\end{proof}
\noindent It follows that $t \models \neg A$ iff $s \not \models A$ for all $s \in T$. The right side is Wright's original condition for negation, and this justifies abbreviating $A \to \bot$ as $\neg A$ (sections \ref{section: Methods}, \ref{subsection: Characteristics (i): Strict finitistic implication}). We note that $\models^V \neg \neg A$ iff $\models^A A$: $\neg \neg A$ stands for practical verifiability, in contrast to $\neg A$ for practical unverifiability (section \ref{section: Methods}). We also have $t \models A \to B$ iff $t \models \neg A \lor \neg \neg B$.

We see that assertibility can be calculated just as classical truth-values, and validity is recursively characterisable.
\begin{proposition}\label{proposition: Assertibility and validity calculations}
  In the context of $S_\infty$,
  \begin{enumerate}
    \item (a) $\models^A A \land B$ iff $\models^A A$ and $\models^A B$. (b) $\models^A A \lor B$ iff $\models^A A$ or $\models^A B$. (c) $\models^A A \to B$ iff $\not \models^A A$ or $\models^A B$.
    \item (a) If $A$ is atomic, then $\not \models^V A$. (b) $\models^V A \land B$ iff $\models^V A $ and $\models^V B$. (c) $\models^V A \lor B$ iff $\models^V A$ or $\models^V B$. (d) $\models^V A \to B$ iff $\models^A A \to B$.
  \end{enumerate}
\end{proposition}
\begin{proof}
  By prevalence property. For (ii, a-c), look at $t_0$.
\end{proof}
\noindent Therefore $\models^V A \to B$ iff $A \to B$ is classically valid; and $\neg A \lor \neg \neg A$, ${\neg \neg A \to A}$ and ${((A \to B) \to A) \to A}$ are valid in $S_\infty$. A general relation with \textbf{CPC} will be given as corollary \ref{corollary: Assertibility is CPC} in terms of the semantics defined in section \ref{section: The strict finitistic Kripke semantics}.

In the introduction (section \ref{section: Methods}), we saw three conditions of practical constructibility. There is a fourth: every statement in strict finitistic reasoning must be `actually weakly decidable' in the sense that it is either assertible or not assertible \cite[p.133]{Wright1982}\footnote{This is rather about practical verifiability. But the first three can also be stated in terms of it: translate (i) into `it is practically verifiable that $\overline{0}$ is constructed', etc.
}.
We suggest that $\models^V \neg A \lor \neg \neg A$ well conforms to this requirement, as it is equivalent to that $\not \models^A A$ or $\models^A A$.

Also, for any assertible $A$ and non-assertible $B$, we have $t_0 \models {\overline{0} = \overline{0} \to A}$ and $t_0 \models^V \neg B$, even if they contain huge NDEs. As noted (section \ref{section: Methods}), these do not reflect the agent's knowledge at $t_0$, but display our judgements about their verification which one must commit to, given the definition of $S_{\infty}$.

$S_\infty$ lacks two principles, the law of excluded middle and Modus Ponens.
\begin{proposition}\label{proposition: Failure of LEM and MP}
  (i) $\not \models_{S_\infty}^V A \lor \neg A$ for some $A$. (ii) For some $A$ and $B$, $\models_{S_\infty}^V B \to A$ and $\models_{S_\infty}^V B$ do not imply $\models_{S_\infty}^V A$.
\end{proposition}
\begin{proof}
  Let $A$ be $\overline{0} = \overline{0}$, and $B$ $(\overline{0} = \overline{0}) \to (\overline{0} = \overline{0})$, and look at $t_0$.
\end{proof}
\noindent (i) is the case even for concrete sentences, unlike in intuitionistic reasoning. Since the nodes represent actual construction steps, there naturally are unverified sentences that will be verified. Similarly, Modus Ponens fails since the conclusion implies $A$ is verified at the beginning. We note that $\models_{S_\infty}^V B \to A$ and $\models_{S_\infty}^V B$, however, imply $\models_{S_\infty}^A A$ for all $A$ and $B$.

These principles in fact hold for the formulas with `stability'. We say that $A$ is \textit{stable in $S_\infty$} if $\neg \neg A \models_{S_\infty}^V A$; and \textit{unstable} otherwise. Namely, assertibility implies validity for the stable formulas. One can easily see $\neg \neg A \models^V A$ iff $\models^V A \lor \neg A$; and, if $A$ is stable, then $\models^V B \to A$ and $\models^V B$ imply $\models^V A$ for all $B$. Although we have not succeeded in recursively characterising the stable formulas, we can present a class with stability. Define
\begin{itemize}
  \item $\ClST ::= \bot \mid \ClForm \to \ClForm \mid \ClST \land \ClST \mid \ClST \lor \ClST$.
\end{itemize}
\begin{proposition}\label{proposition: Stability of ClST}
  Every ClST formula is stable in $S_\infty$.
\end{proposition}
\begin{proof}
  Induction.
\end{proof}

Now we define our models of arithmetic. A tuple $S = \langle T', \leq_{T'}, M', E' \rangle$ is a \textit{model of strict finitistic arithmetic} if (i) it is a finite initial part of a submodel of $S_\infty$ in the sense that (a) $T' \subseteq_{\fin} T$, (b) $\leq_{T'}$, $M'$ and $E'$ are the restrictions of $\leq_T$, $M$ and $E$, respectively, to $T'$, (c) $\langle T', \leq_{T'} \rangle$ is a tree with its root $t'_0$ and (d) each branch is $\{ t \in T \mid t'_0 \leq_T t <_T s \}$ for some $s \in T$, and
\begin{itemize}
  \item [(ii)] for any $p \in \Eq$, if $p \in E'(t)$ for some $t \in T'$, then for any $s \in T'$, there is an $s' \geq_{T'} s$ such that $p \in E'(s')$ (the atomic prevalence condition).
\end{itemize}
Let $\mathcal{S}_{\fin}$ denote the class of all models of arithmetic, and the class $\mathcal{S}$ consist in $\mathcal{S}_{\fin}$ and $S_\infty$.

We carry over the forcing conditions and the semantic notions. For each $S \in \mathcal{S}_{\fin}$, we define a forcing relation between the nodes of $S$ and $\ClForm$ in the same way, and write $\models_S$ instead of $\models_{S_\infty}$. Validity, assertibility, prevalence and semantic consequence in each $S \in \mathcal{S}_{\fin}$ are defined likewise, except that we consider only the nodes of $S$. We will write $\models_S^V A$ etc., respectively. By $\models_{\mathcal{S'}}^V$ etc. for a class $\mathcal{S}' \subseteq \mathcal{S}$, we mean the relations $\models_{S'}^V$ etc. hold, respectively, for all $S' \in \mathcal{S}'$. $A \in \ClForm$ is \textit{stable in $S \in \mathcal{S}$} if $\neg \neg A \models_S^V A$.

Each $S \in \mathcal{S}_{\fin}$ inherits its properties from $S_\infty$. Plainly $\models_S$ persists. The atomic prevalence is extended to all complex formulas.
\begin{proposition}[Full prevalence property]\label{proposition: Full prevalence property}
  For all $S \in \mathcal{S}_{\fin}$ and $A \in \ClForm$, $\models_S^A A$ implies $\models_S^P A$.
\end{proposition}
\begin{proof}
  Induction.
\end{proof}
\noindent The other properties follow from the full prevalence in their respective forms. We note that proposition \ref{proposition: Assertibility and validity calculations} (ii, a) now is that for all atomic $A$, $\models_S^V A$ iff $t'_0 \models_S A$. Also, proposition \ref{proposition: Failure of LEM and MP} takes the following form.
\begin{proposition}
  (i) $\not \models_S^V A \lor \neg A$ for some $S \in \mathcal{S}_{\fin}$ and $A$. (ii) For some $S \in \mathcal{S}_{\fin}$ and $A$ and $B$, $\models_S^V B \to A$ and $\models_S^V B$ do not imply $\models_S^V A$.
\end{proposition}
\begin{proof}
  Take any $S$ with an $A$ unstable in $S$. Let $B$ be $A \to A$.
\end{proof}

Although $S_\infty$ is not officially a model of arithmetic, adding it to $\mathcal{S}_{\fin}$ does not change the totality of the valid formulas. Define for each $t \in T$, $T' := \{ t \mid t_0 \leq_{T} t' \leq_{T} t \}$. Then $T'$ defines a model of arithmetic. We call it the \textit{contraction model $S_\infty \vert t$ of $S_\infty$ with respect to $t$}. Given an $A \in \ClForm$, let $\Eq(A) \subseteq \Eq$ be the set of the equations occurring in $A$. Since $\Eq(A)$ is finite, for each $t \in T$, there is a \textit{contraction node $t' \geq_T t$ with respect to $A$} such that for all $p \in \Eq(A)$, if $\models_{S_\infty}^A p$, then $t' \models_{S_\infty} p$. Let $T_A$ be the set of those nodes.

\begin{proposition}\label{proposition: Model contraction in S}
  (i) For any $t \in T_A$, $\models_{S_\infty}^A A$ iff $t \models_{S_\infty \vert t} A$. (ii) For any $t \in T_A$, $\models_{S_\infty}^V A$ iff $t_0 \models_{S_\infty \vert t} A$. (iii) $\models_{\mathcal{S}}^V A$ iff $\models_{\mathcal{S}_{\fin}}^V A$.
\end{proposition}
\begin{proof}
  (i-ii) Both by induction. Use proposition \ref{proposition: Assertibility and validity calculations} and that ${T_{B \land C}, T_{B \lor C}, }$ ${ T_{B \to C} \subseteq T_B \cap T_C}$. (ii) will use (i). (iii) Immediate from (ii).
\end{proof}

\subsection{The strict finitistic Kripke semantics}\label{section: The strict finitistic Kripke semantics}
In this section, we define the `strict finitistic semantics' to schematically capture validity in $\mathcal{S}$. It is a Kripke-style semantics which resembles the intuitionistic semantics. An \textit{intuitionistic frame} $\langle K, \leq \rangle$ is a nonempty set $K$ partially ordered by $\leq$. Throughout this article, we only consider rooted tree-like intuitionistic frames. We denote a root by $r$. An \textit{intuitionistic model} $\langle K, \leq, v \rangle$ is an intuitionistic frame equipped with a valuation function $v: \Var \to \mathcal{P}(K)$ such that $v(p)$ is closed upwards. An intuitionistic model $\langle K, \leq, v \rangle$ is a \textit{strict finitistic model} if
\begin{enumerate}
  \item $\langle K, \leq \rangle$ is finitely branching,
  \item each branch is of at most countable length,
  \item for each $k \in K$, $\{ p \in \Var \mid k \in v(p) \}$ is finite (the finite verification condition), and
  \item for any $p$, if $k \in v(p)$ for some $k$, then for any $l \in K$, there is an $l' \geq l$ such that $l' \in v(p)$ (the atomic prevalence condition).
\end{enumerate}
Let $\mathcal{W}$ denote the class of all strict finitistic models, which contains infinite as well as finite ones.

Just as the case of $\mathcal{S}_{\fin}$, we carry over the forcing conditions and the semantic notions. We write $\models_W^V A$ etc. for any $W \in \mathcal{W}$. However, (i) the forcing relation is now between the nodes of a model $W$ and the open formulas in $\mathcal{L}_{\SF}$. Therefore this semantics does not deal with construction, but only verification schematically, just as the intuitionistic semantics. (ii) The basis of the definition of the forcing conditions is now $k \models_W p$ iff $k \in v(p)$. (iii) Instead of $\ClST$, we will use the class $\ST \subseteq \Fm$ defined by
\begin{itemize}
  \item $\ST ::= \bot \mid \Fm \to \Fm \mid \ST \land \ST \mid \ST \lor \ST$.
\end{itemize}
It is easy to see that the semantic properties so far are carried over in their respective forms. Particularly, for every $W \in \mathcal{W}$, $\models_W$ persists; and the full prevalence property (proposition \ref{proposition: Full prevalence property}) holds (i.e. $\models_W^A A$ implies $\models_W^P A$).

$\mathcal{W}$ corresponds to $\mathcal{S}$ via substitutions in two senses. We call any mapping $\sigma: \Var \to \Eq$ a \textit{substitution}; and denote the set of the variables occurring in $A \in \Fm$ by $\Var(A)$. $\sigma(A)$ denotes the result of simultaneously substituting all $p \in \Var(A)$ with $\sigma(p)$. If $\Gamma \subseteq \Fm$, $\sigma(\Gamma)$ denotes $\{ \sigma(A) \mid A \in \Gamma \}$. We say $\sigma$ is \textit{faithful to a given $W \in \mathcal{W}$} if for all $p$, $\models_W^A p$ iff $\models_{S_\infty}^A \sigma(p)$.

Firstly, each $W = \langle K, \leq, v \rangle \in \mathcal{W}$ is an abstraction of a part of $S_\infty$. One can confirm that for any substitution $\sigma$ faithful to $W$, there are a $T' \subseteq T$ and a bijection $f: K \to T'$ such that
\begin{enumerate}
  \item for all $k_1, k_2 \in K$, $k_1 \leq k_2$ iff $f(k_1) \leq_{T'} f(k_2)$; and
  \item for all $p$ and $k$, $k \models_W p$ iff $f(k) \models_{S_\infty} \sigma(p)$,
\end{enumerate}
\noindent where $\leq_{T'}$ is the restriction of $\leq_T$ to $T'$. Note that we do not require that $T'$ is closed upwards with respect to $\leq_T$; or that, on both sides of (i), there is no node between the two. We call such a bijection $f_\sigma$. Correspondence (ii) above can be generalised to all formulas.
\begin{proposition}\label{proposition: W into C}
  For any faithful $\sigma$ and $A \in \Fm$, (i) $\models_W^A A$ iff $\models_{S_\infty}^A \sigma(A)$, and (ii) $k \models_W A$ iff $f_\sigma (k) \models_{S_\infty} \sigma(A)$.
\end{proposition}
\begin{proof}
  Both by induction. (ii) will use (i).
\end{proof}

Secondly, the entire class $\mathcal{W}$ captures the validity in $\mathcal{S}$. For each $S = \langle T_0, \leq_{T_0}, M_{\vert T_0}, E_{\vert T_0} \rangle \in \mathcal{S}$ and $\sigma: \Var \to \Eq$, we define the \textit{copy model $W_S \in \mathcal{W}$ of $S$ under $\sigma$} to be $\langle T_0, \leq_{T_0}, v_S \rangle$ where $t \in v_S (p)$ iff $\sigma (p) \in E_{\vert T_0}(t)$.
\begin{proposition}\label{proposition: A copy model}
  For all $t \in T_0$ and $A \in \Fm$, $t \models_{W_S} A$ iff $t \models_S \sigma (A)$.
\end{proposition}
\begin{proof}
  Induction.
\end{proof}

\begin{proposition}
  For any $\Gamma, \Delta \subseteq_{\fin} \Fm$, $\Gamma \models_{\mathcal{W}}^V \Delta$ iff $\sigma (\Gamma) \models_{\mathcal{S}}^V \sigma (\Delta)$ for all $\sigma: \Var \to \Eq$.
\end{proposition}
\begin{proof}
  Fix $\Gamma$ and $\Delta$, and prove the contraposition. ($\implies$) If there are a $\sigma$ and an $S$ such that $\sigma (\Gamma) \not \models_{S}^V \sigma (\Delta)$, then consider the copy $W_S$ of $S$ under $\sigma$. Then by proposition \ref{proposition: A copy model} $\Gamma \not \models_{W_S}^V \Delta$. ($\impliedby$) If there is a $W$ such that $\Gamma \not \models_W^V \Delta$, then consider a substitution faithful to $W$, and use proposition \ref{proposition: W into C}.
\end{proof}

Proposition \ref{proposition: Model contraction in S} can be sharpened: the concept of semantic consequence in $\mathcal{W}$ boils down to that in the class $\mathcal{W}_2 \subseteq \mathcal{W}$ of the 2-node models. We will appeal to this fact when comparing our logic with intermediate logics. Let $W = \langle K, \leq, v \rangle \in \mathcal{W}$. Define for each $k \in K$, the \textit{contraction model $W \vert k \in \mathcal{W}_2$} to be $\langle \{ r, k \}, \leq', v' \rangle$ where $r <' k$, and $v'(p) = v(p) \cap \{ r, k \}$. Given $A \in \Fm$, let $K_A \subseteq K$ be the set of the \textit{contraction nodes $k$} such that for all $p \in \Var(A)$, $\models_W^A p$ implies $k \in v(p)$.

\begin{proposition}[Model contraction]\label{proposition: Model contraction}
  (i) For any $k \in K_A$, $\models_W^A A$ iff $k \models_{W \vert k} A$. (ii) For any $k \in K_A$, $\models_W^V A$ iff $r \models_{W \vert k} A$. (iii) $\models_\mathcal{W}^V A$ iff $\models_{\mathcal{W}_2}^V A$.
\end{proposition}
\begin{proof}
  Similar to proposition \ref{proposition: Model contraction in S}.
\end{proof}
\noindent For each $k \in K$, define the \textit{submodel $W_k \in \mathcal{W}$ generated by $k$} to be ${\langle K', \leq', v' \rangle}$ where $K' = \{ k' \in K \mid k \leq k' \}$, $\leq'$ is the restriction of $\leq$ to $K'$ and $v'(p) = v(p) \cap K'$. Confirm by induction that for all $l \in K'$, $l \models_W A$ iff $l \models_{W_k} A$.
\begin{proposition}\label{proposition: Strong contraction}
  $\Gamma \models_\mathcal{W}^V \Delta$ iff $\Gamma \models_{\mathcal{W}_2}^V \Delta$.
\end{proposition}
\begin{proof}
  ($\impliedby$) We can assume $\Gamma \neq \emptyset$. Suppose $k \in K$ and $k \models_W \bigwedge \Gamma$. Take $k' \geq k$ such that $k' \in K_{\bigwedge \Gamma \land \bigvee \Delta}$, where $\bigvee \Delta$ is $\bot$ if $\Delta = \emptyset$. Since $k \models_W \bigwedge \Gamma$, $\models_{W_k}^V \bigwedge \Gamma$. So $k \models_{W_k \vert k'} \bigwedge \Gamma$ by proposition \ref{proposition: Model contraction} (ii). $\Gamma \models_{\mathcal{W}_2}^V \Delta$ implies $k \models_{W_k \vert k'} \bigvee \Delta$. The rest is easy.
\end{proof}

\subsection{A complete sequent calculus \textbf{SF}}\label{section: A complete sequent calculus SF}
\def\fCenter{\ \vdash\ } 
A sequent calculus \textbf{SF} is defined, and proven to be sound and complete with respect to the strict finitistic semantics. For any finite sequences $\Gamma$ and $\Delta$ in $\Fm$, $\Gamma \vdash \Delta$ is a sequent in \textbf{SF}. In what follows, $\neg \neg \Gamma$ for any $\Gamma$ denotes $\{ \neg \neg A \mid A \in \Gamma \}$. The initial sequents are $\bot \vdash$ and $p \fCenter p$ for all $p \in \Var$. \textbf{SF} has \textbf{LK}'s structural rules. The logical rules are:
\begin{multicols}{2}
\begin{prooftree}
  \Axiom$A, \Gamma \fCenter \Delta$
  \RightLabel{$\land$ L$_1$}
  \UnaryInf$B \land A, \Gamma \fCenter \Delta$
\end{prooftree}
\begin{prooftree}
  \Axiom$\Gamma \fCenter \Delta, A$
  \Axiom$\Gamma \fCenter \Delta, B$
  \RightLabel{$\land$ R}
  \BinaryInf$\Gamma \fCenter \Delta, A \land B$
\end{prooftree}
\end{multicols}

\begin{multicols}{2}
\begin{prooftree}
  \Axiom$A, \Gamma \fCenter \Delta$
  \RightLabel{$\land$ L$_2$}
  \UnaryInf$A \land B, \Gamma \fCenter \Delta$
\end{prooftree}
\begin{prooftree}
  \Axiom$\Gamma \fCenter \Delta, A$
  \RightLabel{$\lor$ R$_1$}
  \UnaryInf$\Gamma \fCenter \Delta, A \lor B$
\end{prooftree}
\end{multicols}

\begin{multicols}{2}
\begin{prooftree}
  \Axiom$A, \Gamma \fCenter \Delta$
  \Axiom$B, \Gamma \fCenter \Delta$
  \RightLabel{$\lor$ L}
  \BinaryInf$A \lor B, \Gamma \fCenter \Delta$
\end{prooftree}
\begin{prooftree}
  \Axiom$\Gamma \fCenter \Delta, A$
  \RightLabel{$\lor$ R$_2$}
  \UnaryInf$\Gamma \fCenter \Delta, B \lor A$
\end{prooftree}
\end{multicols}

\begin{multicols}{2}
\begin{prooftree}
  \Axiom$\Gamma \fCenter \Delta, A$
  \Axiom$B, \Pi \fCenter \Sigma$
  \RightLabel{$\to$ L}
  \BinaryInf$\Gamma, A \to B, \Pi \fCenter \Delta, \neg \neg \Sigma$
\end{prooftree}
\begin{prooftree}
  \Axiom$\Gamma, A \fCenter \neg \neg B, \Delta$
  \RightLabel{$\to$ R}
  \UnaryInf$\Gamma \fCenter A \to B, \neg \neg \Delta.$
\end{prooftree}
\end{multicols}
\noindent We write $\Gamma \vdash_{\SF} \Delta$ to mean that $\Gamma \vdash \Delta$ is derivable in $\SF$.

Proving \textbf{SF}'s soundness is a routine matter.
\begin{proposition}[Soundness]\label{proposition: Soundness}
  $\Gamma \vdash_{\SF} \Delta$ implies $\Gamma \models_{\mathcal{W}}^V \Delta$.
\end{proposition}
\begin{proof}\label{proof: Soundness}
  Induction. Let $W = \langle K, \leq, v \rangle \in \mathcal{W}$, and $k \in K$. {($\to$ L)} Suppose $k \models_W {\bigwedge \Gamma \land (A \to B) \land \bigwedge \Pi}$. Then by the induction hypothesis, $k \models_W \bigvee \Delta \lor A$. If $k \models_W \bigvee \Delta$, we are done. If $k \models_W A$, then there is a $k' \geq k$ such that $k' \models_W B$, and hence $k' \models_W \bigvee \Sigma$. ($\to$ R) Suppose $k \models_W \bigwedge \Gamma$. Use that $\models_W^V \neg A \lor \neg \neg A$. If $\models_W^V \neg A$, then $k \models_W A \to B$. If $\models_W^V \neg \neg A$, then there is a $k' \geq k$ such that $k' \models_W A$, and hence $k' \models \neg \neg B \lor \bigvee \Delta$. The rest is easy.
\end{proof}

We note that $A \vdash_{\SF} A$, $A \vdash_{\SF} \neg \neg A$ and $\neg \neg \neg A \vdash_{\SF} \neg A$ for all $A \in \Fm$. Also, $A \to B \dashv \vdash_{\SF} \neg A \lor \neg \neg B$, and hence $\vdash_{\SF} \neg A \lor \neg \neg A$. Admissible are
\begin{multicols}{2}
\begin{prooftree}
  \Axiom$A, \Gamma \fCenter \Delta$
  \RightLabel{$\neg \neg$ L}
  \UnaryInf$\neg \neg A, \Gamma \fCenter \neg \neg \Delta$
\end{prooftree}
\begin{prooftree}
  \Axiom$\Gamma \fCenter \Delta, A$
  \RightLabel{$\neg \neg$ R}
  \UnaryInf$\Gamma \fCenter \Delta, \neg \neg A.$
\end{prooftree}
\end{multicols}

We can reproduce the stability result of the ST formulas in this system. Confirm that derivable are (i) $\neg \neg (A \land B) \vdash \neg \neg A \land \neg \neg B$, (ii) $\neg \neg (A \lor B) \vdash \neg \neg A \lor \neg \neg B$ and (iii) $\neg \neg (A \to B) \vdash A \to B$.
\begin{proposition}
  $\neg \neg S \vdash_{\SF} S$ for all $S \in \ST$.
\end{proposition}
\begin{proof}
  Induction on $S$. $\bot$'s case is easy. Use (i-iii) above.
\end{proof}
\noindent Here, $\neg \neg A \vdash_{\SF} A$ iff $\vdash_{\SF} A \lor \neg A$ for all $A \in \Fm$. Therefore $\vdash_{\SF} S \lor \neg S$ for all $S \in \ST$.  $A, A \to S \vdash_{\SF} S$ is also plain. Inferences involving an ST formula resemble \textbf{LK}.
\begin{proposition}
  For any $S \in \ST$, the following $\LK$ rules are admissible.
\begin{multicols}{2}
\begin{prooftree}
  \Axiom$\Gamma \fCenter \Delta, A$
  \Axiom$S, \Pi \fCenter \Sigma$
  \RightLabel{$\to$ L$^*$}
  \BinaryInf$\Gamma, A \to S, \Pi \fCenter \Delta, \Sigma$
\end{prooftree}
\begin{prooftree}
  \Axiom$\Gamma, S \fCenter B, \Delta$
  \RightLabel{$\to$ R$^*$}
  \UnaryInf$\Gamma \fCenter S \to B, \Delta$
\end{prooftree}
\end{multicols}
\end{proposition}
\begin{proof}
  ($\to$ L$^*$) Apply ($\to$ L) to the left premise and $S \vdash S$. Then cut $\neg \neg S$ using $\neg \neg S \vdash_{\SF} S$. Finally cut $S$ using the right premise. ($\to$ R$^*$) Cut $S$ from $\vdash S, \neg S$ and the premise. Use $\neg S \vdash_{\SF} S \to B$ and $B \vdash_{\SF} S \to B$.
\end{proof}

Now let us prove the completeness of \textbf{SF}. Given $\Gamma \subseteq \Psi \subseteq \Fm$, we say that $\Gamma$ is a \textit{prime theory with respect to $\Psi$} if, for all $A, B \in \Psi$, (i) $A \lor B \in \Gamma$ implies $A \in \Gamma$ or $B \in \Gamma$, and (ii) $\Gamma \vdash_{\SF} A$ implies $A \in \Gamma$. Also, we let $\Var(\Gamma) = \bigcup_{A \in \Gamma} \Var(A)$; and $\Fm(\Gamma)$ be the set of all formulas built only from $\Var(\Gamma)$. For $\Gamma, \Delta \subseteq \Fm$, we say $\langle \Gamma, \Delta \rangle$ is \textit{consistent} if for any $\Gamma' \subseteq_{\fin} \Gamma$ and $\Delta' \subseteq_{\fin} \Delta$, $\Gamma' \not \vdash_{\SF} \Delta'$; otherwise, it is \textit{inconsistent}.

The proof progresses as follows. First, let $\Gamma, \Delta \subseteq_{\fin} \Fm$ such that $\Gamma \not \vdash_{\SF} \Delta$ be given. Then we construct a prime theory.
\begin{lemma}[Extension lemma]
  There is a prime theory $\Gamma^* \supseteq \Gamma$ with respect to $\Fm(\Gamma \cup \Delta)$ such that $\langle \Gamma^*, \Delta \rangle$ is consistent.
\end{lemma}
\begin{proof}
  Define $\Gamma^*$ as follows. First, $\Gamma_0 := \Gamma$. Then $\Gamma_0 \not \vdash_{\SF} \Delta$. Let $\langle A_n \lor B_n \rangle_{n \in \mathbb{N}}$ enumerate with infinite repetition all disjunctive formulas in $\Fm(\Gamma \cup \Delta)$. Let $\Gamma_n$ with $\Gamma_n \not \vdash_{\SF} \Delta$ be given. If $\Gamma_n \vdash_{\SF} A_n \lor B_n$, then $\Gamma_n, A \not \vdash_{\SF} \Delta$ or $\Gamma_n, B \not \vdash_{\SF} \Delta$, since \textbf{SF} has (Cut). If the former, define $\Gamma_{n+1}$ to be $\Gamma_n \cup \{ A \}$; otherwise $\Gamma_n \cup \{ B \}$. Then $\Gamma_{n+1} \not \vdash_{\SF} \Delta$. Define $\Gamma^* := \bigcup_{n \in \mathbb{N}} \Gamma_n$.
  
  $\Gamma^* \subseteq \Fm(\Gamma \cup \Delta)$ is plain by induction. (i) If $A \lor B \in \Gamma^*$, then let $n$ be the least of $m$ such that $\Gamma_m \vdash_{\SF} A \lor B$. Then there is an $n' \geq n$ such that we consider $A \lor B$ at the $n'$-th step, and hence $A \in \Gamma_{n'+1}$ or $B \in \Gamma_{n'+1}$. (ii) If $\Gamma^* \vdash_{\SF} A$, then $\Gamma^* \vdash_{\SF} A \lor A$. So, by the same argument, we conclude $A \in \Gamma^*$. Finally, for all $n \in \mathbb{N}$, $\Gamma_n \not \vdash_{\SF} \Delta$ by induction. So $\langle \Gamma^*, \Delta \rangle$ cannot be inconsistent.
\end{proof}

\begin{lemma}
  For any $A \in \Fm(\Gamma \cup \Delta)$, either $\neg A \in \Gamma^*$ or $\neg \neg A \in \Gamma^*$.
\end{lemma}
\begin{proof}
  Immediate from $\vdash_{\SF} \neg A \lor \neg \neg A$ and $\langle \Gamma^*, \Delta \rangle$ being consistent.
\end{proof}
\noindent Now define $W = \langle \{ k, k' \}, \leq, v \rangle$ by $k < k'$, $k \in v(p)$ iff $p \in \Gamma^*$ and $k' \in v(p)$ iff $\neg \neg p \in \Gamma^*$. Then $W \in \mathcal{W}_2$, since $\Var(\Gamma \cup \Delta)$ is finite. $k$ reflects $\Gamma \not \vdash_{\SF} \Delta$.
\begin{lemma}[Truth lemma]
  For all $A \in \Fm(\Gamma \cup \Delta)$, (i) $\neg \neg A \in \Gamma^*$ iff $k' \models_W A$; and (ii) $A \in \Gamma^*$ iff $k \models_W A$.
\end{lemma}
\begin{proof}
  Both by induction. Use the properties of $\Gamma^*$. (i) The base case is trivial. For ($\land$), use that $\neg \neg (A \land B) \dashv \vdash_{\SF} \neg \neg A \land \neg \neg B$. ($\lor$) is similar. For ($\to$), use that $\neg \neg (A \to B) \dashv \vdash_{\SF} \neg A \lor \neg \neg B$. (ii) The base case, ($\land$) and ($\lor$) are easy. For ($\to$), appeal to (i).
\end{proof}
\noindent The completeness is now established by appealing to the 2-node model.
\begin{proposition}[Completeness]
  $\Gamma \models_{\mathcal{W}}^V \Delta$ implies $\Gamma \vdash_{\SF} \Delta$.
\end{proposition}
\begin{proof}
  Contraposition. By the preceding lemmas, we can assure there is a $W \in \mathcal{W}_2$ with $k \models_W A$ for all $A \in \Gamma$, and $k \not \models_W B$ for all $B \in \Delta$.
\end{proof}

\subsection{Relations with intermediate logics}\label{section: Relations with intermediate logics}
In this section we will compare \textbf{SF} with intermediate logics. We will establish that $\HT \subsetneq \textbf{SF} \subsetneq \textbf{CPC}$; and then show how strict finitistic models can be seen as nodes of an intuitionistic model. \textbf{HT} is known to be characterised by the class of the 2-node intuitionistic models, which we will denote by $\mathcal{W}_{2i}$. We use $\Vdash$ for the intuitionistic forcing relation. 

In comparing with \textbf{CPC} and \textbf{HT}, we will look at $\mathcal{W}_2$ (cf. proposition \ref{proposition: Strong contraction}). Since all $W \in \mathcal{W}_2$ have the isomorphic frames, we fix one, $F := {\langle \{k, k'\}, \leq \rangle}$ with $k < k'$, and consider the class $\mathcal{V}_2$ of the valuations on $F$. If a model $W = \langle \{k, k'\}, \leq, v \rangle$ is considered, we will write $\models_v^V A$, etc. instead of $\models_W^V A$, etc. We will also be writing $\models_{\mathcal{V}_2}^V A$ etc. A formula is \textit{stable in a $v \in \mathcal{V}_2$} if it is stable in the model $\langle \{ k, k'\}, \leq, v \rangle$. Stability in $\mathcal{V}_2$ is understood accordingly.

It is easy to see $\textbf{SF} \subsetneq \textbf{CPC}$: indeed, all rules of \textbf{SF} are admissible in \textbf{LK}, while $\not \vdash_{\SF} A \lor \neg A$ by soundness. But further we can establish a sufficient condition for a classical theorem to be a theorem of $\SF$. Completeness of \textbf{SF} implies that $A$ is stable in $\mathcal{V}_2$ iff $\neg \neg A \vdash_{\SF} A$. For each $W = \langle K, \leq, v \rangle \in \mathcal{W}$, we write $\cl(v)$ for the \textit{classical part of $v: \Fm \to \{ 0, 1 \}$} defined by $\cl(v)(p) = 1$ iff $\models_W^P p$. By $\vdash_{\textbf{X}} A$, we mean that $A$ is a theorem of an intermediate logic $\textbf{X}$.
\begin{proposition}
  If $\vdash_{\CPC} A$ and $A$ is stable in $\mathcal{V}_2$, then $\vdash_{\SF} A$.
\end{proposition}
\begin{proof}
  If $\vdash_{\CPC} A$, then $\cl(v)(A) = 1$ for all $v \in \mathcal{V}_2$. Therefore $\models_{\mathcal{V}_2}^V \neg \neg A$, and hence $\vdash_{\SF} \neg \neg A$ by completeness of \textbf{SF}. Then $\vdash_{\SF} A$, since $A$ is stable.
\end{proof}
\begin{corollary}\label{corollary: Assertibility is CPC}
  $\models_\mathcal{W}^A A$ iff $\vdash_{\CPC} A$.
\end{corollary}
\begin{proof}
  Use that $\models_\mathcal{W}^A A$ iff $\vdash_{\SF} \neg \neg A$; and $\neg \neg A \in \ST$.
\end{proof}

Next, we will see that $\HT \subsetneq \SF$. $\SF \not \subseteq \HT$ is plain, since Peirce's law does not belong to \textbf{HT}. To show $\HT \subseteq \SF$, we will appeal to the intuitionistic semantics. Let $\mathcal{V}_{2i}$ be the class of the intuitionistic valuations on $F$. Then $\mathcal{V}_2 \subseteq \mathcal{V}_{2i}$. Therefore it makes sense to intuitionistically evaluate a formula in $F$ with a strict finitistic valuation $v \in \mathcal{V}_2$. We may write $k \Vdash_v A$ etc. similarly in this context. The only difference in the definition of $\Vdash$ is that $l \Vdash_v B \to C$ iff for all $l' \geq l$, $l' \Vdash_v B$ implies $l' \Vdash_v C$. Validity is defined in the same way as $\models_v$; and we will write $\Vdash_v^V A$ etc. In this setting, to say $\HT$ is characterised by $\mathcal{W}_{2i}$ is to say that for all $A$, $\vdash_{\HT} A$ iff $\Vdash_{\mathcal{V}_{2i}}^V A$. We note that therefore $\vdash_{\HT} A$ implies $\Vdash_{\mathcal{V}_2}^V A$, since $\mathcal{V}_2 \subseteq \mathcal{V}_{2i}$.
\begin{proposition}
  (i) $k' \Vdash_v A$ iff $k' \models_v A$. (ii) $k \Vdash_v A$ implies $k \models_v A$. (iii) $\Vdash_{\mathcal{V}_2}^V A$ implies $\models_{\mathcal{V}_2}^V A$.
\end{proposition}
\begin{proof}
  (i, ii) by induction. (i) is plain since both evaluations are classical. (ii) ($\to$) Use persistence and (i). (iii) Immediate from (ii).
\end{proof}

\begin{proposition}\label{proposition: HT subseteq SF}
  $\HT \subseteq \SF$.
\end{proposition}
\begin{proof}
  Use that $\vdash_{\HT} A$ implies $\models_{\mathcal{V}_2}^V A$, and apply completeness of \textbf{SF}.
\end{proof}

Finally, we will show how strict finitistic finite models can be viewed as nodes of an intuitionistic model that satisfies the finite verification condition. Let $\mathcal{W}_{\fin} \subseteq \mathcal{W}$ be the class of the finite models, $\mathcal{U} \subseteq \mathcal{W}_{\fin}$ be at most countable, $\preceq \, \subseteq \mathcal{U}^2$ and for each $W \in \mathcal{U}$, $W = \langle K_W, \leq_W, v_W \rangle$. Then $\preceq$ is a \textit{generation order on $\mathcal{U}$} if $W_1 \preceq W_2$ iff (i) $K_{W_1} \subseteq K_{W_2}$, (ii) for all $k, k' \in K_{W_1}$, $k \leq_{W_1} k'$ iff $k \leq_{W_2} k'$ and (iii) $v_{W_1}(p) \subseteq v_{W_2}(p)$. Plainly, a generation order is a partial order. $\langle \mathcal{U}, \preceq \rangle$ is a \textit{generation structure} (\textit{g-structure}) if it is an intuitionistic frame and $\preceq$ is a generation order. We only consider rooted tree-like intuitionistic frames. Let $\mathcal{G}$ be the class of all g-structures. For any $G = \langle \mathcal{U}, \preceq \rangle \in \mathcal{G}$, we write $R_G \, (\in \mathcal{U})$ for $G$'s root; $r_G \, (\in K_{R_G})$ for $R_G$'s root; and $K_G$ for $\{ \langle W, k \rangle \mid W \in \mathcal{U} \land k \in K_W \} $.

A g-structure $\langle \mathcal{U}, \preceq \rangle$ specifies how an agent's cognitive abilities increase, and represents how their actual verification of basic facts, such as concrete equations, proceeds under the increasement. Each $W \in \mathcal{U}$ typically represents the same agent from different generations, and a later generation is associated with a larger amount power and more verification steps.

Given a $G = \langle \mathcal{U}, \preceq \rangle \in \mathcal{G}$, we induce a mapping $v_G: \Var \to \mathcal{P}(K_G)$ from the valuations of $\mathcal{U}$'s elements: $\langle W, k \rangle \in v_G(p)$ iff $k \in v_W(p)$. We can show $v_G(p)$ is closed upwards with respect to both $G$ and each $W \in \mathcal{U}$.
\begin{proposition}
  If $\langle W, k \rangle \in v_G(p)$, then $\langle W', k' \rangle \in v_G(p)$ for all $W' \succeq W$ and $k' \, (\in K_{W'}) \geq_{W'} k$.
\end{proposition}
\begin{proof}
  By persistence of $v_W$ in each $W$ and definition of $v_G$.
\end{proof}
\noindent We define the \textit{generation forcing relation} $\Vvdash_G \, \subseteq K_G \times \Fm$ extending $v_G$ as follows.
\begin{enumerate}
  \item $W, k \Vvdash_G p$ iff $\langle W, k \rangle \in v_G(p)$;
  \item $W, k \Vvdash_G A \land B$ iff $W, k \Vvdash_G A$ and $W, k \Vvdash_G B$;
  \item $W, k \Vvdash_G A \lor B$ iff $W, k \Vvdash_G A$ or $W, k \Vvdash_G B$; and
  \item $W, k \Vvdash_G A \to B$ iff for any $W' \succeq W$ and $k' \, (\in K_{W'}) \geq_{W'} k$, if $W', k' \Vvdash_G A$, then there is a $k'' \, (\in K_{W'}) \geq_{W'}    k'$ such that $W', k'' \Vvdash_G B$.
\end{enumerate}
$\Vvdash_G$ thus defined formalises the agent's actual verification under the increasing power. $W, k \Vvdash_G A \to B$ means that if $A$ is verified with some extension in power, then so is $B$ with the same power; $W, k \Vvdash_G \neg A$ means $A$ is practically unverifiable with any extension. If $A$ is atomic, $W, k \Vvdash_G A$ means that the agent has verified $A$ by, and knows that $A$ at step $k$ of generation $W$. Otherwise, it only reflects our judgement.

$A \in \Fm$ is \textit{valid in a $G = \langle \mathcal{U}, \preceq \rangle \in \mathcal{G}$} if $W, k \Vvdash_G A$ for all $\langle W, k \rangle \in K_G$. We write $\Vvdash_G^V A$ for this; and $\Vvdash_{\mathcal{G}}^V A$ for that $\Vvdash_G^V A$ for all $G \in \mathcal{G}$. One can verify by induction that $\Vvdash_G$ persists both in $G$ and each $W$ just as $v_G$. Therefore $\Vvdash_G^V A$ iff $R_G, r_G \Vvdash_G A$. Since each $W$ is a strict finitistic model, we have the prevalence property of $\Vvdash_G$ inside of them.
\begin{proposition}[Prevalence in generation]
  If $W, k \Vvdash_G A$ for some $k \in K_W$, then for any $l \in K_W$, there is an $l' \, (\in K_W) \geq_W l$ such that $W, l' \Vvdash_G A$.
\end{proposition}
\begin{proof}
  Induction on $A$. Use persistence of $\Vvdash_G$.
\end{proof}

\begin{corollary}\label{corollary: Assertibility of implication in g-structure is validity}
  $W, r_G \Vvdash_G A \to B$ iff $W, k \Vvdash_G A \to B$ for some $k \in K_W$.
\end{corollary}
\begin{proof}
  ($\impliedby$) Follows from prevalence in generation.
\end{proof}
\noindent We note that $\models_W^P$ where $W \in \mathcal{U}$ does not persist in $G$. One may be too powerless to verify ($\models_W^P \neg p$), but may verify later ($\models_{W'}^P p$ with $W \preceq W'$).

Let $\mathcal{J}$ denote the class of the intuitionistic models that satisfy the finite verification condition. We see that a $G = \langle \mathcal{U}, \preceq \rangle \in \mathcal{G}$ induces a corresponding $I \in \mathcal{J}$, and vice versa. Given $G$, we define a mapping $v: \Var \to \mathcal{P}(\mathcal{U})$ by $W \in v(p)$ iff there is a $k \in K_W$ such that $W, k \Vvdash_G p$; and let $I_G := \langle \mathcal{U}, \preceq, v \rangle$. Then $v(p)$ is closed upwards in $G$, and since each $W \in \mathcal{U}$ is finite, $I_G \in \mathcal{J}$ follows. So $W \Vdash_{I_G} A$ makes sense for each $A \in \Fm$ and $W \in \mathcal{U}$. We mean by $\Vdash_{I_G}^V A$ that $W \Vdash_{I_G} A$ for all $W \in \mathcal{U}$. $\Vdash_{I_G}$ persists in $G$, and therefore $\Vdash_{I_G}^V A$ iff $R_G \Vdash_{I_G} A$. The correspondence between $\Vdash_{I_G}$ and $\Vvdash_G$ can be generalised to all complex formulas.
\begin{proposition}\label{proposition: G to J transformation}
  $W \Vdash_{I_G} A$ iff there is a $k \in K_W$ such that $W, k \Vvdash_G A$.
\end{proposition}
\begin{proof}
  Induction. ($\land$) ($\implies$) If $W \Vdash_{I_G} B \land C$, then $W, k \Vvdash_G B$ and $W, l \Vvdash_G C$ for some $k, l \in K_W$. So by prevalence and persistence, $W, k' \Vvdash_G B \land C$ for some $k' \in K_W$ such that $k \leq_W k'$.
  
  ($\to$) ($\implies$) Suppose $W \Vdash_{I_G} B \to C$. We will prove $W, r_G \Vvdash_G B \to C$. Let $W \preceq W'$, $k \in K_{W'}$ and $W', k \Vvdash_G B$. Then $W' \Vdash_{I_G} B$, and hence $W' \Vdash_{I_G} C$ by the supposition. Therefore $W', l \Vvdash_G C$ for some $l \in K_{W'}$. So by prevalence, $W', k' \Vvdash_G C$ for some $k' \geq_{W'} k$. ($\impliedby$) Let $W, k \Vvdash_G B \to C$, $W \preceq W'$ and $W' \Vdash_{I_G} B$. Then $W', l \Vvdash_G B$ for some $l \in K_{W'}$. Since $k \in K_{W'}$, $W', k' \Vvdash_G B$ for some $k' \geq_{W'} k$ by prevalence. So $W', k'' \Vvdash C$ for some $k'' \geq_{W'} k'$. Use the induction hypothesis.
\end{proof}
\noindent Conversely, let $I = \langle U^*, \preceq^*, v^* \rangle \in \mathcal{J}$ be given. Then, for each $U \in U^*$, we can obtain $W_U = \langle K_U, \leq_U, v_U \rangle \in \mathcal{W}_{\fin}$ by letting $K_U = \{ U' \in U^* \mid U' \preceq^* U \}$, $U' \leq_U U''$ iff $U' \preceq^* U''$ and $v_U(p) = v^*(p) \cap K_U$. Define $\mathcal{U} := \{ W_U \mid U \in U^* \}$; $\preceq \, \subseteq \mathcal{U}^2$ by $W_{U_1} \preceq W_{U_2}$ iff $U_1 \preceq^* U_2$; and $G_I := \langle \mathcal{U}, \preceq \rangle$. Then $G_I \in \mathcal{G}$. $I$ and $G_I$ correspond in the following sense.
\begin{proposition}\label{proposition: J to G transformation}
  For all $A \in \Fm$ and $U \in U^*$, $U \Vdash_I A$ iff $W_U, U \Vvdash_{G_I} A$.
\end{proposition}
\begin{proof}
  Induction. ($\to$) ($\implies$) Let $U \Vdash_I B \to C$, and $W_{U'}, U'' \Vvdash_{G_I} B$ with $U \preceq U'' \preceq U'$. Then, since $W_{U'}, U' \Vvdash_{G_I} B$ by persistence, $U' \Vdash_I B$. Therefore $U' \Vdash_I C$. So $W_{U'}, U' \Vvdash_{G_I} C$, ($\impliedby$) Suppose $W_U, U \Vvdash_{G_I} B \to C$, $U \preceq^* U'$ and $U' \Vdash_I B$. Then $W_{U'}, U' \Vvdash_{G_I} B$. Since $U'$ is the maximum of $\langle K_{U'}, \leq_{U'} \rangle$, $W_{U'}, U' \Vvdash_{G_I} C$ by the supposition.
\end{proof}

These correspondences between $\mathcal{G}$ and $\mathcal{J}$ yield that the totality of the formulas valid in all $G \in \mathcal{G}$ coincides with \textbf{IPC}. We note that as easily seen, $\vdash_{\IPC} A$ iff $\Vdash_I^V A$ for all $I \in \mathcal{J}$.
\begin{proposition}\label{proposition: Validity in G is IPC}
  For all $A \in \Fm$, (i) $\Vvdash_{\mathcal{G}}^V A$ iff (ii) $\Vdash_{I_G}^V A$ for all $G \in \mathcal{G}$ iff (iii) $\vdash_{\IPC} A$.
\end{proposition}
\begin{proof}
  (i $\! \implies \!$ ii) Let $G \in \mathcal{G}$. $\Vvdash_G^V A$ implies $R_G, r_G \Vvdash_G A$, and hence $R_G \Vdash_{I_G} A$.
  
  (ii $\implies$ iii) Let $H = \langle U^*, \preceq^*, v^* \rangle \in \mathcal{J}$, and $R$ be $H$'s root. Then we have $\Vdash_H^V A$ iff $R \Vdash_H A$ iff $W_R, R \Vvdash_{G_H} A$ iff $W_R \Vdash_{I_{G_H}} A$. (ii) implies the last.
  
  (iii $\implies$ i) Prove by induction that for all $A$, $\vdash_{\IPC} A$ implies $\Vvdash_{\mathcal{G}}^V A$. The basis and ($\land$) are trivial. ($\lor$) By disjunction property of $\IPC$. ($\to$) does not depend on the induction hypothesis. Assume $\Vdash_I^V B \to C$ for all $I \in \mathcal{J}$. Let $G = { \langle \mathcal{U}, \preceq \rangle } \in \mathcal{G}$. By corollary \ref{corollary: Assertibility of implication in g-structure is validity}, $R_G, r_G \Vvdash_G B \to C$ iff $R_G, k \Vvdash_G B \to C$ for some $k \in K_{R_G}$. By the assumption we have $R_G \Vdash_{I_G} B \to C$. Apply proposition \ref{proposition: G to J transformation}.
\end{proof}

These results may be of philosophical interest. $\HT \subsetneq \textbf{SF} \subsetneq \textbf{CPC}$ (cf. corollary \ref{corollary: Assertibility is CPC}, proposition \ref{proposition: HT subseteq SF}) means that $\SF$ is too strong to be properly intermediate. Yet it has an intimate relation with $\IPC$: propositions \ref{proposition: J to G transformation} and \ref{proposition: Validity in G is IPC} seem to endorse the conceptual relation between strict finitism and intuitionism mentioned in section \ref{subsection: Characteristics (ii): Relation with intuitionism}. The prevalence conditions (cf. propositions \ref{proposition: Prevalence property} and \ref{proposition: Full prevalence property}) may naturally be suspected to be responsible, but we do not yet exactly know how they are contributing.

\section{Ending remarks: Further topics}\label{section: Ending remarks: Further topic}

This article only presents our first attempt to classically formalise Wright's strict finitistic logic, and we must leave at least two topics for further investigations. (i) Wright pointed out that the atomic prevalence cannot be assumed in general (cf. section \ref{section: Methods}), and his sketch already included the predicate part. The semantics needs be extended in these two directions. Without the prevalence conditions, the connection with \textbf{CPC} will be lost, and $\neg A$ will separate from $A \to \bot$, although Peirce's law and $\neg A \lor \neg \neg A$ may remain valid. Investigating what the logic looks like might provide an answer to the philosophical question mentioned above. Meanwhile, the predicate part of our semantics should involve quantification ranging over all objects in the domain, not only those after one node, due to the global nature of negation. We are preparing an article for these purposes, with a sound and complete proof system with respect to the extended semantics.

(ii) We maintained Wright's forcing conditions, and as part of our classical idealisation, accepted any finite lengths of time-gap in implication (cf. section \ref{subsection: Characteristics (i): Strict finitistic implication}). Indeed, from the literature it appears hardly possible to philosophically motivate a specific number as a maximum length. However, developing a theory of restricted implication may be of interest. Intuitionistic implication is strict finitistic implication with time-gap 0. If we associate $\omega$ to $\SF$ because the consequent can come within any $n < \omega$ steps, then it might be reasonable to associate $1$ to $\IPC$, since $0 < 1$. With a theory of implication with various lengths, then we might see a gradation of logics between \textbf{SF} and \textbf{IPC}.

\section*{Acknowledgements}
\begin{itemize}
  \item This study was funded under the name of `Graduate Scholarship for Degree Seeking Students' by Japan Student Services Organisation.
  \item I would like to thank Rosalie Iemhoff for her everyday, most insightful and helpful pieces of advice. I am also grateful to Amirhossein Akbar Tabatabai for his penetrating comments: indeed, it was he who first suggested the relation with $\IPC$. I sincerely thank an anonymous referee for extremely helpful comments and suggestions. Some of the rudimentary versions of the present article were presented at the OZSW Conference 2020, the 3rd Workshop on Proof Theory and its Applications in 2021 and the Dutch Logic PhD Day 2022. This study was conducted under the doctoral supervision by Professor Rosalie Iemhoff at Utrecht University.
\end{itemize}






\end{document}